\documentclass[11pt,reqno]{amsproc}
\usepackage[margin=1in]{geometry}
\usepackage{amsmath, amsthm, amssymb}
\usepackage{times, esint,stackrel}
\usepackage{enumitem}
\usepackage{color}
\usepackage{hyperref}

\newcommand{\be}{\begin{equation}}
\newcommand{\ee}{\end{equation}}
\newcommand{\bea}{\begin{eqnarray}}
\newcommand{\eea}{\end{eqnarray}}

\newtheorem{thm}{Theorem}

\newtheorem{lemma}{Lemma}

\newtheorem{rem}{Remark}

\newcommand{\ve}{{\varepsilon}}
\newcommand{\rmd}{{\rm d}}
\newcommand{\bx}{{ {x} }}

\newcommand{\bu}{{{u}}}
\newcommand{\ol}[1]{\mkern 1.5mu\overline{\mkern-1.5mu#1\mkern-1.5mu}\mkern 1.5mu}

\newcommand{\Dlim}{{\mathcal{D}'}\mbox{-}\lim}

\title[Onsager's Singularity Theory]{An Onsager Singularity Theorem for \\  Leray Solutions of Incompressible Navier-Stokes}
\author{Theodore D. Drivas and Gregory L. Eyink}
\address{Department of Mathematics, Princeton University, Princeton, NJ 08544}
\email{tdrivas@math.princeton.edu}
\address{Department of Applied Mathematics and Statistics,  Johns Hopkins University, Baltimore, MD 21218}
\email{eyink@jhu.edu}

\date{today}
\begin{document}

\begin{abstract}
We study in the inviscid limit  the global energy dissipation of Leray solutions of incompressible 
Navier-Stokes on the torus ${\mathbb T}^d$, assuming that the solutions have norms for
Besov space $B^{\sigma,\infty}_3({\mathbb T}^d),$ $\sigma\in (0,1],$ that are bounded in the $L^3$-sense in time, 
uniformly in viscosity.  We establish an upper bound on energy dissipation of the form 
$O(\nu^{(3\sigma-1)/(\sigma+1)}),$ vanishing as $\nu\to0$  if $\sigma>1/3.$ A consequence 
is that Onsager-type ``quasi-singularities'' are required in the Leray solutions, even if the total energy 
dissipation vanishes in the limit $\nu\to 0$, as long as it does so sufficiently slowly.  We also give two sufficient 
conditions which guarantee  the existence of limiting weak Euler solutions $u$ which satisfy a local energy 
balance with possible anomalous dissipation due to inertial-range energy cascade in the Leray solutions.
For $\sigma\in (1/3,1)$ the anomalous dissipation vanishes and the weak Euler solutions may be 
spatially ``rough'' but conserve energy.   
\end{abstract}

\maketitle

\section{Introduction}

In a 1949 paper on turbulence in incompressible fluids \cite{O49}, L. Onsager announced a result that spatial H\"older exponents 
$\leq 1/3$ are required of the velocity field for anomalous turbulent dissipation (that is, energy dissipation non-vanishing 
in the limit of zero viscosity). Onsager's original statement and most subsequent work 
\cite{GLE94,CET94,DR00,CCFS08,LS09,LS10,LS12,I16,BLSV17} have involved the conjecture 
that the velocity field in the limit of infinite Reynolds number is a weak (distributional) solution of the incompressible Euler 
equations. In this short paper we show that the arguments employed to prove Onsager's claim about weak Euler solutions 
apply as well to Leray's solutions of the incompressible Navier-Stokes equation and can be used to prove a theorem that 
``quasi-singularities'' are required in those solutions in order to account for anomalous energy dissipation. 
In fact, such {consequences} follow even if the energy dissipation is vanishing in the limit of zero viscosity, 
{as long as it goes to zero as slowly as $\sim \nu^\alpha$ for some $\alpha\in (0,1)$.} 
{In that case, we show that the Navier-Stokes solutions cannot have Besov norms, above a critical smoothness 
$\frac{1+\alpha}{3-\alpha},$ which are bounded uniformly in viscosity.} 
%and not tending to a positive value. 
{This observation is important because {empirical studies (e.g. see Remark 4 below) cannot distinguish in principle}
between a dissipation rate which is 
independent of viscosity and one which is vanishing sufficiently slowly. Our results thus considerably strengthen the conclusion 
that quasi-singularities are necessary to account for the enhanced energy dissipation rates observed in turbulent flow.} 
No assumption need be made in our proof about existence of limiting Euler solutions, but weak Euler solutions do arise 
as $\nu\to 0$ limits of the Leray solutions if some further natural conditions are satisfied. 

Let $\bu^\nu\in  L^\infty([0,T];L^2({\mathbb T}^d))\cap L^2([0,T]; H^1({\mathbb T}^d))$ for $\nu>0$ be Leray solutions 
of the incompressible Navier-Stokes equations 
satisfying 
\bea\label{NSE}
\partial_t \bu^\nu + \nabla \cdot (\bu^\nu\otimes \bu^\nu) \!\! &=& \!\! -\nabla p^\nu + \nu \Delta \bu^\nu +f^\nu,\\
\nabla \cdot u^\nu \!\!&=& \!\!0,
\eea
in the sense of distributions on $\mathbb{T}^d\times [0,T],$ with solenoidal initial conditions 
$u^\nu|_{t=0}=u_0^\nu\in  L^2({\mathbb T}^d)$ and  solenoidal body forcing $f^\nu\in L^2([0,T]; L^2({\mathbb T}^d))$. 
A fundamental property of these solutions, first obtained by Leray \cite{L34}, is the global energy inequality, which states 
that viscous energy dissipation cannot exceed the loss of energy by the flow plus the energy input by external force. 
This property may be reformulated as a global balance of kinetic energy:
\bea\label{viscousDiss}
\int_0^T\int_{\mathbb{T}^d} \! \varepsilon[\bu^\nu] \  \rmd \bx \rmd t
= \frac{1}{2}\int_{\mathbb{T}^d} \! | \bu_0^\nu|^2\rmd \bx
-\frac{1}{2}\int_{\mathbb{T}^d} \! |\bu^\nu(\cdot,T)|^2\rmd \bx
+\int_0^T\int_{\mathbb{T}^d} \! \bu^\nu\cdot f^\nu\  \rmd \bx \rmd t,
\eea
for almost every $T\geq 0$, where the total energy dissipation rate is
\be
\varepsilon[\bu^\nu] := \nu |\nabla \bu^\nu|^2+ D[\bu^\nu]
\label{viscousDiss2} \ee
 with $D[\bu^\nu]$ a non-negative distribution (Radon measure) that represents dissipation due to possible Leray singularities. 
 See Duchon-Robert \cite{DR00} and the proof of our Lemma 1.  {Our main result is then:}
 
\begin{thm}\label{theorem1}
Let $\bu^\nu\in L^\infty([0,T];L^2({\mathbb T}^d))\cap L^2([0,T]; H^1({\mathbb T}^d))$ for $\nu>0$ 
be any Leray solutions of incompressible Navier-Stokes equations on $\mathbb{T}^d\times [0,T]$ with initial data $
\bu_0^\nu\in B_2^{\sigma,\infty}({\mathbb T}^d)$, and forcing $f^\nu\in L^2([0,T]; B_2^{\sigma,\infty}({\mathbb T}^d))$  for some $\sigma\in(0,1].$
  Suppose  that: 
\be\label{lowerBndDiss}
\int_0^T\int_{\mathbb{T}^d} \! \varepsilon[\bu^\nu] \  \rmd \bx \rmd t
\geq \nu^\alpha L(\nu), \quad \quad  \alpha\in [0,1)
\ee
where $L: \mathbb{R}^+\to \mathbb{R}^+$ is a function slowly-varying at $\nu=0$ in the sense of Kuramata 
\cite{BGT89}, i.e. so that 
$\lim_{\nu\to 0} L(\lambda\nu)/L(\nu)=1$ for any $\lambda>0.$
Then, for any $\epsilon>0$, the family $\{u^\nu\}_{\nu>0}$ of Leray solutions cannot have norms 
$\|u^\nu\|_{L^{3}([0,T];B_3^{\sigma_\alpha+\epsilon, \infty}({\mathbb T}^d))}$ with $\sigma_\alpha:= \frac{1+\alpha}{3-\alpha}\in[1/3,1)$
that are bounded uniformly in $\nu>0$.
\end{thm}

Theorem \ref{theorem1} follows easily from the following lemma:  

%%%%%%%%

\begin{lemma}\label{lemma1}
Let $\{u^\nu\}_{\nu>0}$ be a family of Leray solutions  with $\sigma$, $
\bu_0^\nu$, and  $f^\nu$  as in Theorem \ref{theorem1}.
Assume that  $\bu^\nu\in L^{3}([0,T];B_3^{\sigma,\infty}({\mathbb T}^d))$ with all the above Besov norms
bounded, uniformly in viscosity.  Then, for a.e. $T\geq 0$, the energy dissipation is bounded for some $\nu$-independent constant $C$ by:
 \be\label{noAnom}
\int_0^T\int_{\mathbb{T}^d} \! \varepsilon[\bu^\nu] \  \rmd \bx \rmd t
\leq C \nu^{\frac{3\sigma-1}{\sigma+1}}.
\ee
 \end{lemma}
 
%%%%%%%%

To see that Theorem \ref{theorem1}  follows from Lemma \ref{lemma1}, note that if for any $\epsilon>0,$
$\bu^\nu \in L^{3}([0,T];B_3^{\sigma_\alpha+\epsilon, \infty}({\mathbb T}^d))$ with norms bounded uniformly in viscosity, then the inequality 
\eqref{noAnom} together with \eqref{lowerBndDiss} implies:
\be\label{contradictionIneq}
L(\nu)  \leq C \nu^{\epsilon\frac{(3-\alpha)^2}{4+\epsilon (3-a)}}.
\ee
Since $\alpha \in [0,1)$, the exponent in the power-law on the righthand side of \eqref{contradictionIneq} is positive.  
This obviously leads to a contradiction since $\lim_{\nu\to 0} \nu^{-p}L(\nu)=+\infty$ for $L$ slowly 
varying at $\nu=0$ and for any $p>0.$
\vspace{3mm}

{In the context of Lemma \ref{lemma1},} we note that that if $\sigma\in [1/3,1]$ then Theorem 6.1 of \cite{CCFS08} implies that $D[u^\nu]=0$ 
and energy dissipation arises entirely from viscosity.  {The proof of this fact for $\sigma>1/3$ and fixed $\nu>0$ follows easily  
by the Constantin-E-Titi commutator argument \cite{CET94} for weak solutions, after taking into account the Leray-Hopf regularity 
$L^2(0,T;H^1(\mathbb{T}^d))$.} {We conjecture that our Theorem 1 is optimal for space dimensions $d>2$ in the sense 
that, for some $\alpha\in [0,1),$ there should exist sequences of Leray solutions of Navier-Stokes $u^\nu$ for  
$\nu>0$ that are uniformly bounded in $L^{3}([0,T];B_3^{\sigma_\alpha-\epsilon, \infty}({\mathbb T}^d))$ with any $\epsilon>0$ and 
for which the lower bound \eqref{lowerBndDiss} on dissipation holds as an asymptotic equality for $\nu\to 0$. The case $d=2$
is different, because of the absence of vortex-stretching. This implies strong bounds on enstrophy for Leray solutions in $d=2,$ 
even with initial vorticity $\omega_0\in L^p$ only for $p<2,$ and an essential improvement of the energy dissipation bounds
in our Lemma 1 for $d=2$ \cite{Ches16}.}

%If furthermore $\sigma\in (1/3,1]$ the upper bound (\ref{noAnom}) on the viscous dissipation vanishes for $\nu\to 0$ as a power-law $\nu^\alpha$ with exponent $\alpha= \frac{3\sigma-1}{\sigma+1}\in (0,1]$. 

%{Finally, the upper bound given by Lemma \ref{lemma1} on the dissipation rate is not sharp for decaying fluid motion in two-dimensions if $\omega_0\in L^p$ for $p\in (1,2)$.  In fact, if  we assume that $\omega_0\in L^p({\mathbb T}^2)$ then, by the maximum principle  $\omega_t\in L^p({\mathbb T}^2)$ for all $t\geq 0$ with norms uniformly-in-$\nu$ bounded.  This directly implies that  $u_t\in  B_p^{1,p}({\mathbb T}^2) \subseteq B_p^{1,\infty}({\mathbb T}^2)$, so that  $u\in L^\infty(0,T;B_3^{s_p,\infty}({\mathbb T}^2))$ with $s_p:= 5/3-2/p$ by the Besov embedding theorem on $\mathbb{T}^2$.  Requiring $s_p\geq 0$ restricts us to considering $p\geq 6/5$.  The conditions of Lemma \ref{lemma1} then apply and we find that the dissipation vanishes at least as fast as $(\rm {const}.)\nu^{\beta_p}$ with $\beta_p:=3( (4p-6)/(8p-6))$.  Note that $\beta_p\geq 0$ only for $p\in[3/2,2)$.  On the other hand, the rate obtained in the proof of Theorem 2 of \cite{Ches16} is $\beta'_p:= 2(p-1)/p>0$ for all $p\in (1,2)$ and since $\beta'_p>\beta_p$ for all such $p$, this shows our argument gives a suboptimal bound when applied in 2D with rough initial conditions. }

\begin{rem}\label{Besov}
{\rm 
The main condition on uniform Besov regularity in Lemma \ref{lemma1} is physically natural.
{The Besov space $B_p^{\sigma,\infty}(\mathbb{T}^d)$ is made up of measurable functions $f:\mathbb{T}^d\to \mathbb{R}^d$ which are finite in the norm
\be
\| f\|_{B_p^{\sigma,\infty}(\mathbb{T}^d)}:= \| f\|_{L^p(\mathbb{T}^d)} 
+ \sup_{r\in(0,1]^d}\frac{\|f(\cdot+r) - f(\cdot)\|_{\mathbb{T}^d}}{|r|^\sigma}
\ee
for $p\geq 1$ and $\sigma\in (0,1)$. {See \cite{ST87}, section 3.5.} 
These spaces can be equivalently explained in a way more familiar to fluid dynamicists by using structure functions.}
The $p$th--order structure functions $S_p^\nu(r)$ of spatial velocity-increments 
$\delta u^\nu(r;x,t):=u^\nu(x+r,t)-u^\nu(x,t)$ may be defined as usual by $S_p^\nu(r,t):=\langle |\delta u^\nu(r,t)|^p\rangle,$
where $\langle\cdot\rangle$ denotes space average over $x\in {\mathbb T}^d.$ The velocity field belongs to the Besov
space $B_p^{\sigma,\infty}({\mathbb T}^d)$ for $p\geq 1,$ $\sigma\in (0,1)$ at time $t$ if and only if 
\be
\langle |u^\nu(\cdot,t)|^p\rangle<C_0(t), \qquad S_p^\nu(r,t) \leq C_1(t)\left|\frac{r}{\ell_0}\right|^{\zeta_p}, \ \forall |r|\leq \ell_0 %\quad\quad \quad\zeta_p> \left(\frac{1+\alpha}{3-\alpha}\right)p
\label{space-struc-fun} \ee 
with $\zeta_p=\sigma p$ and then the optimal constants $C_0(t),$ $C_1(t)>0$ in these upper bounds define a norm for the Besov space $B_p^{\sigma,\infty}({\mathbb T}^d)$ by the 
identification $\|u^\nu(\cdot,t)\|_{B_p^{\sigma,\infty}({\mathbb T}^d)}:=[C_0(t)+C_1(t)]^{1/p}.$ E.g. see \cite{GLE95}.
Here any choice of length-scale $\ell_0>0$ defines the same function space $B_p^{\sigma,\infty}({\mathbb T}^d)$
but for a physical identification of the constant $C_1(t)$ as the ``amplitude'' of an inertial-range scaling law,
one must take $\ell_0$ to be the integral-length of the turbulent flow and independent of $\nu>0.$
The uniform boundedness of the family $\{u^\nu\}_{\nu>0}$ in $L^{p}([0,T];B_p^{\sigma,\infty}({\mathbb T}^d))$ 
is equivalent to the condition that coefficients $C_0(t),$ $C_1(t)$ independent of $\nu>0$ should exist so that the 
bounds (\ref{space-struc-fun}) are satisfied for a.e. $t\in [0,T]$ and $\int_0^T dt\ [C_0(t)+C_1(t)]<\infty.$
The Theorem \ref{theorem1} and Lemma \ref{lemma1} apply {\it a fortiori} to solution spaces 
$L^p([0,T],B^{\sigma,\infty}_p({\mathbb T}^d))$ with 
any $p\geq 3$ and not only to $p=3.$ As a consequence, energy dissipation vanishing with $\nu\to 0$ as slowly as 
(\ref{lowerBndDiss}) (or possibly not vanishing at all for $\alpha=0$), implies 
$\zeta_p\leq \left(\frac{1+\alpha}{3-\alpha}\right)p\,$ for $p\geq 3$
as a constraint on possible structure-function scaling exponents in the inertial-range of any turbulent flow
with enhanced dissipation of the form (\ref{lowerBndDiss}). 
This inequality is a precise statement on ``quasi-singularities'' in the sequence of Leray solutions, in order to be consistent 
with the observed slow decrease of energy dissipation as $\nu\to 0.$  The Navier-Stokes solutions (barring possible true, Leray-type singularities) 
are spatially $C^\infty$ for any $\nu>0,$ but they cannot possess smoothness of the form (\ref{space-struc-fun}) that is uniform in viscosity.  {The primary physical motivation of our result is turbulence in space dimensions $d>2$,  where a forward energy 
cascade is expected.  However our theorem has some implications even for $d=2$.  
For example, reference \cite{Ches16} considers Navier-Stokes solutions with initial vorticity $\omega_0\in L^p(\mathbb{T}^2)$, $p\in (1,2]$
and obtains an upper bound on energy dissipation of the form $(\rm {const}.)\nu^{\alpha_p}$ for $\alpha_p:= \frac{2(p-1)}{p}\in (0,1],$
vanishing as $\nu\to 0$.  
If this is the actual scaling of the dissipation for $p<3/2$, the Onsager critical value of $p$ for $d=2,$ then our Theorem  \ref{theorem1} 
implies that the family $\{u^\nu\}_{\nu>0}$ cannot be uniformly bounded in 
${L^{3}([0,T];B_3^{\sigma_{\alpha_p}+\epsilon, \infty}({\mathbb T}^2))}$ with 
$\sigma_{\alpha_p}:= \frac{3p-2}{p+1}\in (1/2,1)$.}

\begin{rem}\label{rem:Bc}
A small but useful technical improvement of Theorem \ref{theorem1} can be easily provided by sharpening the spaces considered. First, recall that energy conservation for weak solutions of the Euler equations holds provided that $u\in B_3^{1/3, c_0}(\mathbb{T}^d)$, a subspace of $B_3^{1/3, \infty}(\mathbb{T}^d)$ that can be defined as follow
\begin{equation}\label{def:Bc}
B_p^{\sigma, c_0}(\mathbb{T}^d)=\left\{f\in L^p(\mathbb{T}^d):~\lim_{|r|\to 0}\frac{\|f(\cdot+r) - f(\cdot)\|_{L^p(\mathbb{T}^d)}}{|r|^\sigma}=0\right\}.
\end{equation}
See \cite{CCFS08}.
Note that $B_p^{\sigma', \infty}(\mathbb{T}^d)\subset B_p^{\sigma, c_0}(\mathbb{T}^d)\subset B_p^{\sigma, \infty}(\mathbb{T}^d)$ for any $\sigma'>\sigma$. Define also
\be\label{def:Bct}
L^q(0, T; B_p^{\sigma, c_0}(\mathbb{T}^d))=\left\{f\in L^q(0, T; L^p(\mathbb{T}^d)):~\lim_{|r|\to 0}\frac{\|f(\cdot+r) - f(\cdot)\|_{L^q(0, T; L^p(\mathbb{T}^d)}}{|r|^\sigma}=0\right\}.
\ee
Theorem \ref{theorem1} then holds in a form in which one replaces all instances of $B_p^{\sigma,\infty}$ with $B_p^{\sigma,c_0}$ and the conclusion reads that the family $\{u^\nu\}_{\nu>0}$ of Leray solutions cannot have norms 
$\|u^\nu\|_{L^{3}([0,T];B_3^{\sigma_\alpha, c_0}({\mathbb T}^d))}$ with $\sigma_\alpha:= \frac{1+\alpha}{3-\alpha}\in[1/3,1)$.    Note that the spaces  $B_p^{\sigma, c_0}$  allow us to remove the ``$\epsilon$" appearing in the theorem statement. The proof is almost identical and therefore omitted.   We are grateful to the anonymous referee for this remark.
\end{rem}

We emphasize again that we do not need to assume that any ``singular'' or ``rough'' Euler solutions exist in order to draw these 
conclusions. However, under reasonable additional conditions, weak Euler solutions will exist as inviscid limits of the Leray solutions.
For example:
 }
 \end{rem} 
 
 \begin{thm}\label{theorem2}
Let $\bu^\nu\in L^\infty([0,T];L^2({\mathbb T}^d))\cap L^2([0,T]; H^1({\mathbb T}^d))$ be 
any Leray solutions of incompressible Navier-Stokes equations with $\nu>0$ on $\mathbb{T}^d\times [0,T],$ 
for initial data $
\bu_0^\nu\in L^2({\mathbb T}^d)$ and forcing $f^\nu\in L^2([0,T]; L^2({\mathbb T}^d)),$  and assume either: 
\begin{quotation}
(i) For some $\sigma\in(0,1]$ the family 
$\{u^\nu\}_{\nu>0}$ is uniformly bounded in $L^{3}([0,T];B_3^{\sigma,\infty}({\mathbb T}^d)),$ { and that $f^{\nu}\to f$ strongly in $L^2([0,T];L^2({\mathbb T}^d))$ as $\nu\to 0^+$}. Let
 $u$ then be any strong limit of a subsequence $u^{\nu_k}\in L^3([0,T];L^3({\mathbb T}^d))$.
\end{quotation} 

\noindent or 

\begin{quotation}
(ii) \ $\bu^\nu\in L^{3}([0,T];L^3({\mathbb T}^d))$ with norms bounded uniformly in viscosity and furthermore, 
that  weak convergence as $\nu\to 0$ holds for a full-measure set of times: 
\be u^\nu(\cdot,t)\stackrel[L^3]{}{\rightharpoonup} u(\cdot,t), 
\quad 
(u^\nu \otimes u^\nu)(\cdot,t) \stackrel[L^{3/2}]{}{\rightharpoonup} (u\otimes u)(\cdot,t), \quad f^\nu(\cdot,t)\stackrel[L^2]{}{\rightharpoonup} f(\cdot,t)
\quad \mbox{ a.e. $t\in [0,T]$.} 
\label{wkconv} \ee 
\end{quotation} 

Then $u$ is a weak Euler solution  which also satisfies, in the sense of distributions, the energy balance 
\be \partial_t\left(\frac{1}{2}|u|^2\right)+\nabla\cdot\left[\left(\frac{1}{2}|u|^2+p\right)u\right] = -D[u] + u\cdot f \label{Ebal} \ee 
on ${\mathbb T}^d\times [0,T]$, with $D[u]$ the distributional limit of nonlinear ``energy flux" for the Leray solutions: 
\be D[u] :=\ \stackrel[\ell\to 0]{}{\Dlim}\ \stackrel[\nu\to 0]{}{\Dlim}\Pi_\ell[u^\nu].
\label{flux-anom} \ee 
See definition (\ref{Piell}) below.
 Furthermore, under the condition (i) 
\be D[u]= \ \stackrel[\nu\to 0]{}{\Dlim} \varepsilon[\bu^\nu] \label{flux-anom2}, \ee
where total dissipation measure $\varepsilon[\bu^\nu]$ for Leray solutions is defined in (\ref{viscousDiss2}), 
{and $u\in L^{3}([0,T];B_3^{\sigma-\epsilon, c_0}({\mathbb T}^d))$ for any $\epsilon>0.$ 
Thus, $D[u]=0$ and local energy conservation holds when $\sigma\in(1/3,1]$.} 
\end{thm}
 
% \blue{
%Note that under the condition (i) of Theorem \ref{theorem2}, the limiting weak Euler solutions have Besov regularity 
% $u\in L^{3}([0,T];B_3^{\sigma,\infty}({\mathbb T}^d))$.  See Remark \ref{limitBesov}. Thus $D[u]=0$ and energy conservation 
% holds if $\sigma\in(1/3,1]$.}

  \begin{rem}\label{IsettCVrem}
{\rm We owe the first condition of Theorem \ref{theorem2} to P. Isett \cite{Phil1}, reproduced here with permission.  In particular, 
he pointed out that uniform boundedness of a family of weak Navier-Stokes solutions $\{\bu^\nu\}_{\nu>0}$ in $L^{2}([0,T];B_2^{\sigma,\infty}(\mathbb{T}^d))$  guarantees 
strong pre-compactness in $L^{2}(\mathbb{T}^d\times [0,T])$ 
by the Aubin-Lions-Simon Lemma 
(see also \cite{PhilPrePrint}).  
Isett pointed out to us \cite{Phil} that the uniform boundedness assumed in Lemma \ref{lemma1} allows such an argument 
also for $p=3$. In the physical application this means that if energy dissipation is bounded below 
as in (\ref{lowerBndDiss}) but if also $\{u^\nu\}_{\nu>0}$ is uniformly bounded in 
$L^{3}([0,T];B_3^{\sigma_\alpha-\epsilon,\infty}({\mathbb T}^d))$ for any $\epsilon>0$, then a limit Euler solution $u$ 
will exist. {Moreover, the limit will possess some spatial Besov regularity with exponent $\sigma_\alpha-\epsilon$ 
but not {\it a priori} with a higher exponent $\sigma_\alpha+\epsilon$ for any $\epsilon>0.$ %See Remark \ref{limitBesov} below.
}

The second part of the theorem slightly generalizes recent results of Constantin \& Vicol \cite{CV17} for wall-bounded domains 
$\Omega$. There, it is proved that if $u^\nu\rightharpoonup u$ weakly in $L^2(\Omega)$ for a.e. $t$ and {if 
a second-order structure function ${\mathcal S}_2^\nu(r)$ defined as in our Remark 1 (but also time-averaged)
satisfies an inertial-range scaling bound like \eqref{space-struc-fun}}, 
then $u$ is a weak solution to the Euler equations (see Theorem 3.1 of \cite{CV17}). {Recently, the condition on weak-convergence at a.e. time $t$ was removed in \cite{DN18} in favor of assuming a structure function bound within a more precise ``inertial range".  Also, as pointed out in \cite{CV17}, Remark 3.4, this condition may be removed by assuming 
a bound on the {\it space-time structure function} defined by 
${\mathcal S}_p^\nu(r,s):=\langle\!\langle |\delta u^\nu(r,s)|^p\rangle\!\rangle,$ 
where $\delta u^\nu(r,s;x,t)=u^\nu(x+r,t+s)-u^\nu(x,t)$ are space-time increments  
and where $\langle\!\langle\cdot\rangle\!\rangle$ denotes the space-time average over $(x,t)\in \Omega\times [0,T].$}   
{Specifically, it is assumed in \cite{CV17} for $p=2$ that}
{
\be\label{spacetime}
 \langle\!\langle | u^\nu|^p\rangle\!\rangle \leq  {\mathcal C}_0 \qquad {\mathcal S}_p^\nu(r,s) \leq {\mathcal C}_1\left[\left|\frac{r}{\ell_0}\right|
+\left|\frac{s}{t_0}\right|\right]^{\zeta_p}, \ \  \ \forall\ \eta(\nu)\leq  |r|\leq \ell_0,  \ \, \tau(\nu)\leq s\leq t_0
\ee
with some $\zeta_p>0$, {$\nu$-independent}
constants ${\mathcal C}_0,$ ${\mathcal C}_1>0$, and any scales $\eta(\nu)$, $\tau(\nu)$ converging to  $0$ as $\nu\to 0$.  If the bound \eqref{spacetime} is assumed to hold for  $\eta(\nu)= \tau(\nu)\equiv 0$, then  \eqref{spacetime} 
is the uniform regularity statement $ \sup_{\nu>0}\|\bu^\nu\|_{B_2^{\sigma,\infty}(\Omega\times [0,T])}<\infty$ for some $\sigma\in (0,1)$ and compactness {in $L^{2}(\Omega\times [0,T])$ with the strong topology} is immediately implied}
{ by the Kolmogorov--Riesz theorem \cite{HH10}.}  
Thus, subsequences $\nu_k\to 0$ always exist for which $u^{\nu_k}\to u$ strongly 
in $L^2$ and the limit function $u$ is automatically a weak Euler solution. 
We could likewise replace the condition (ii) {at each time slice in Theorem \ref{theorem2}} 
by the assumption that {\eqref{spacetime} holds for $p=3$, i.e. uniform third-order space-time structure function bounds in the inertial range, 
{and take $u$ to be any weak limit point of $u^\nu\in L^3(0,T;L^3(\mathbb{T}^d))$.}
 Furthermore, the limiting Euler solution inherits the space-time regularity $u\in B_3^{\sigma,\infty}(\Omega\times [0,T])$ by an argument similar to that in Remark \ref{limitBesov}.}

An earlier theorem giving conditions for convergence of Navier-Stokes solutions to weak Euler solutions 
satisfying a global energy inequality is proved in the work of Chen \& Glimm \cite{CG12}.  Their sufficient 
conditions involve the time-average energy spectrum, or $p=2,$ because all terms of the energy balance that 
are cubic in the velocity vanish when integrated over space.}

\end{rem}

 \begin{rem}\label{remEvid}
 {\rm It is worthwhile to review briefly here the empirical evidence regarding the global energy dissipation 
 rate in {boundary-free} turbulent flow. Numerical simulations of Fourier-truncated Navier-Stokes dynamics by pseudo-spectral 
 method in a periodic box correspond mostly closely to the conditions of our Theorem \ref{theorem1}. Free-decay 
 simulations with body-force $f^\nu=0$ such as \cite{BO95,TBS02} do show a non-vanishing energy flux 
 in the inertial-range, consistent  with $D[u]>0$ as defined in (\ref{flux-anom}), but there seems to have been no systematic 
 study of the dependence of space-average $\langle \ve^\nu(t) \rangle$ upon $\nu=1/Re$ in such simulations.
 Forced simulations with very smooth (large-scale) forces $f^\nu$ \cite{KRS98,KIYIU03} provide the best evidence 
 for a space-time average $\langle \ve^\nu\rangle$ which is nearly independent of $\nu=1/Re$ as $Re\to\infty.$  
 These simulations are nominally ``long-time steady-states'' with $T\to\infty,$ but in practice the time-averages are performed 
 only over several large-eddy turnover times, so that our Theorem \ref{theorem1} applies.  Given the data 
 plotted in Fig.~1 of \cite{KRS98} or Fig.~3 of \cite{KIYIU03} a reasonable inference is that the dissipation rate does not vanish
 as $Re\to\infty$, or vanishes only weakly with viscosity.  Accepting this as an empirical fact, our Theorem \ref{theorem1}
for $p=\infty$ implies that Onsager's prediction of H\"older exponents $h\leq 1/3$ \cite{O49} remains valid as a statement 
about ``quasi-singularities'' of Leray solutions. If any of the reasonable conditions in the Theorem 
\ref{theorem2} hold as well, then Onsager's conjecture on {weak} Euler solutions remains true, even if 
the dissipation rate is vanishing weakly as $\nu\to 0$. In the latter case the Euler solutions {may be}
spatially ``singular'' or ``rough'', but conserve energy. It should be emphasized that the Euler singularities inferred by this argument 
need not develop in finite time from smooth initial data. A standard practice in such numerical simulations is the 
initialization $u^\nu(\cdot,0)=u^{\nu'}(\cdot,T')$ of the simulation at high $Re$ by the final state at time $T'$ of a smaller 
Reynolds-number $Re'<Re$  simulation performed at lower resolution, interpolated onto the finer grid of the 
$Re$-simulation (e.g. see p.L21 of \cite{KIYIU03}). This practice of ``nested'' initialization means that initial conditions $u^\nu(\cdot,0)$
have Kolmogorov-type spectra over increasing ranges of scales as $\nu$ decreases and do not correspond 
to uniformly smooth initial data. 

Similar remarks apply to studies of dissipation rates {in boundary-free flows} by laboratory experiment. 
The most common experiments study turbulence produced downstream of wire-mesh grids in wind-tunnels 
or turbulent wakes generated by flows past other solid obstacles, such as plates, cylinders, etc. 
\cite{KRS84,PKW02,DLKA17}. 
These experiments measure the time-averaged kinetic energy $(1/2)\langle |u^\nu(x,\cdot)|^2\rangle$ at distances 
$x$ down-stream of the obstacle. If the data are reinterpreted by ``Taylor's hypothesis'' as space-averages 
$(1/2)\langle |u^\nu(\cdot,t)|^2\rangle$ at times $t=x/U,$ with $U$ the mean flow velocity, then these studies 
yield the space-average dissipation rate  $\langle \ve^\nu(t) \rangle$ by time-differentiation. The data plotted in 
\cite{KRS84,PKW02,DLKA17} again provide corroboratory evidence that $\langle \ve^\nu(t) \rangle$ is nearly independent 
of $\nu=1/Re$ as $Re$ increases. These experiments are obviously not in the space-periodic framework 
of our Theorem \ref{theorem1}. Ignoring the effects of walls in the wind-tunnel, at some distance from the 
turbulent wake, these flows might be regarded as contained in some large box with zero velocities at the wall
(and thus periodic). However, the creation of the turbulence by flow past solid obstacles implies that these 
experiments are closer to the setting of \cite{CV17}, with vorticity fed into the flow by viscous boundary layers 
that detach from the walls. Since the boundary layers become thinner as $\nu=1/Re$ decreases, the initial
data of these experiments also cannot be considered to be smooth uniformly in $\nu>0.$} \end{rem}
%%%%%%%%

\begin{rem}
{ In light of the discussion in Remark \ref{remEvid}, theoretically incorporating the effects of solid confining walls  is of great practical 
importance. The experimental observations are rather different for wall-bounded turbulence, such as seen as in pipes, channels, closed 
containers, etc., than those reviewed above for boundary-free flows. Energy dissipation in confined turbulent flows with rough walls 
tends to constant values for $Re\gg 1,$ whereas energy dissipation in flows with smooth walls is generally observed to vanish with increasing 
$Re,$ yet much more slowly than the laminar rate $\sim 1/Re.$ For example, see the study \cite{CCDDT97} whose results are typical.   
Recently, there have been a number of papers proving Onsager-type theorems on necessary conditions for anomalous dissipation 
by weak solutions of the Euler equations on domains with solid boundaries \cite{BT18,DN18Ons,BTW18}.  The statements of energy dissipation 
are slightly more involved due to the fact that assumptions need to be made both in the interior and near the walls. The results of Drivas and 
Nguyen \cite{DN18Ons}, which focus on vanishing viscosity limits of Leray solutions, may be  modified to provide results in the same spirit 
of our Theorem \ref{theorem1}.  In particular, \S 2.4  of \cite{DN18Ons} provides a connection between the physical energy dissipation and 
coarse-grained fluxes as in Lemma \ref{lemma}. If one supposes that the energy dissipation  is lower bounded as in \eqref{lowerBndDiss} 
and introduces quantitative versions of the near-wall assumptions (i.e. impose how rapidly the velocity itself of the near-wall dissipation 
vanishes within a viscous boundary layer as viscosity tends to zero), then Theorem 2 and 3 of \cite{DN18Ons} can translated into constraints 
on uniform interior Besov regularity and boundary-layer behavior of Leray--Hopf solutions.  Detailed implications are left for future investigation.}
\end{rem}

The proof our Lemma \ref{lemma1} will be based on the same method employed by Constantin-E-Titi
\cite{CET94} to prove the original Onsager statement for weak Euler solutions, by means of a spatial mollification. 
Specifically, let $G$ be a \emph{standard mollifier}, with $G\in D(\mathbb{T}^d),$ $G\geq 0,$ and also
 $\int_{\mathbb{T}^d} G(r) \rmd r=1.$ Without loss of generality, we can assume that ${\rm supp}(G)$ is contained in the 
 Euclidean unit ball in $d$ dimensions.  Define the dilatation $G_\ell(r)=\ell^{-d} G(r/\ell)$ and space-reflection 
 $\check{G}(r)=G(-r)$. For any $v\in D'(\mathbb{T}^d),$ we define its \emph{coarse-graining at scale $\ell$} 
 by \be \ol{v}_\ell=\check{G}_\ell *v\in C^\infty(\mathbb{T}^d). \label{cg-def} \ee 
 %%%%%%%%
Then, we have the following:
\begin{lemma}\label{lemma}
Let initial data $\bu_0^\nu\in L^2(\mathbb{T}^d)$,  forcing $f^\nu\in L^2([0,T]; L^2(\mathbb{T}^d))$ and  $\bu^\nu$ be 
corresponding Leray solutions of the incompressible Navier-Stokes equations on $\mathbb{T}^d\times [0,T]$ for $\nu>0$.  
{Then, the following local resolved energy balance  holds {for any $\ell>0$}, for every $x\in {\mathbb T}^d$ and a.e. $t\in [0,T]$
\be\label{resolvedEnergyBal}
\partial_t\left( \frac{1}{2} |\overline{(\bu^\nu)}_{\ell}|^2 \right) +\nabla \cdot J_{\ell}^\nu =  -\Pi_{\ell}[\bu^\nu] -\nu |\nabla \ol{(\bu^\nu)}_{\ell}|^2 + \ol{(\bu^\nu)}_{\ell}\cdot \ol{(f^\nu)}_{\ell},\ee
with 
\be\label{J-def}  
\quad \quad  J_{\ell}^\nu:= \left(\frac{1}{2} |\overline{(\bu^\nu)}_{\ell}|^2+ \overline{(p^\nu)}_{\ell} \right)\overline{(\bu^\nu)}_{\ell}  + \overline{(\bu^\nu)}_{\ell}\cdot  {\tau}_{\ell} (\bu^\nu,\bu^\nu) -\nu\nabla \left(\frac{1}{2} |\overline{(\bu^\nu)}_{\ell}|^2\right) .
\ee
where the coarse-graining cumulant is defined by ${\tau}_{\ell} (g,h):=  \ol{(g\otimes h)}_\ell-\ol{g}_\ell\otimes \ol{h}_\ell\ $ for 
$g,h\in L^2(\mathbb{T}^d,\mathbb{R}^d),$ the trace is denoted by ${\tau}_{\ell} (g\,;h):={\rm Tr}\,\tau_\ell(g,h)$ and where
\be\label{Piell}
\Pi_{\ell}[\bu^\nu]  :=- \nabla\overline{(\bu^\nu)}_{\ell} : {\tau}_{\ell} (\bu^\nu,\bu^\nu).
\ee}
Furthermore, 
 for a.e. $T\geq 0$ and for any standard mollifier $G$ {and any $\ell>0$}, we have:
\bea \nonumber
\int_0^T\int_{\mathbb{T}^d} \! \varepsilon[\bu^\nu] \  \rmd \bx \rmd t
&=&  \int_0^{T}\!\int_{\mathbb{T}^d} \!  \Pi_{\ell}[\bu^\nu]\  \rmd \bx \rmd t+  \int_0^{T}\!\int_{\mathbb{T}^d} \! \nu  |\nabla \ol{(\bu^\nu)}_\ell|^2\  \rmd \bx \rmd t\\ \nonumber
&&+\    \frac{1}{2} \int_{\mathbb{T}^d} \! \tau_{\ell}(u_0^\nu;u_0^\nu)\ \rmd \bx -  \frac{1}{2} \int_{\mathbb{T}^d} \! \tau_{\ell}(u^\nu(\cdot,T); u^\nu(\cdot,T))\\
&&+\ \int_0^{T}\!\int_{\mathbb{T}^d} \!  \tau_{\ell}(\bu^\nu; f^\nu)\  \rmd \bx \rmd t \label{fourFifthsLaw}
\eea
\end{lemma}
\noindent The key ingredient of the proof of Lemma \ref{lemma1} is a simple exact formula derived in \cite{CET94} which 
expresses the ``energy flux'' $\Pi_{\ell}[\bu^\nu]$ in terms of velocity increments.  Our relation (\ref{flux-anom}) can thus
be interpreted as an extension of the celebrated Kolmogorov 4/5th--law to infinite Reynolds-number limits of Leray solutions. 

%%%%%%%%

\section{Proofs}

\begin{proof}[Proof of Lemma \ref{lemma}]
Any {Leray} weak solution $u^\nu$ of Navier-Stokes satisfies point-wise in $x\in \mathbb{T}^d$
and distributionally in $t\in [0,T]$ the coarse-grained equations
\be\label{weak}
\partial_t \ol{(\bu^\nu)}_{\ell} + \nabla \cdot [\ol{(\bu^\nu \otimes \bu^\nu)}_{\ell}] = - \nabla \ol{(p^\nu)}_{\ell}  
+ \nu \Delta  \ol{(\bu^\nu)}_{\ell}+ \ol{(f^\nu)}_{\ell}. 
\ee
{We use here the velocity-pressure formulation of Leray solutions, with pressure
$p^\nu \in W^{-1,\infty} (0,T;L^2(\mathbb{T}^d))$ (e.g. see Theorem V.1.4 of \cite{BF13}). 
The $d$ equations \eqref{weak} can then be obtained by mollifying the Navier--Stokes equations with 
(non-solenoidal) test functions $\varphi_i,$ $i=1,2, \dots, d,$ of the form $\varphi_i(r,t):=\psi(t) G_\ell(r-x)e_i$ 
where $\psi\in C_0^\infty((0,T)),$ $G\in C^\infty(\mathbb{T}^d)$, and $e_i$ is the unit vector in the $i$th 
coordinate direction.} 

{We now show that the classical time derivative of  $\ol{(\bu^\nu)}_{\ell}(x,t)$ exists for every $x\in\mathbb{T}^d$ and a.e. $t\in [0,T]$.  See also Prop. 2 of \cite{D18}. }
Since Leray solutions satisfy $u^\nu \in L^\infty([0,T]; L^2(\mathbb{T}^d))$,  then for every $x\in \mathbb{T}^d$
\begin{eqnarray}
&& \|\nabla \cdot [\ol{(\bu^\nu \otimes \bu^\nu)}_{\ell}](x,\cdot)\|_{L^\infty([0,T])}
\leq \frac{1}{\ell}\|(\nabla G)_\ell\|_\infty \|u\|^2_{L^\infty([0,T]; L^2(\mathbb{T}^d))},\cr 
&& \|\nu \Delta  \ol{(\bu^\nu)}_{\ell}(x,\cdot)\|_{L^\infty([0,T])}
\leq \frac{\nu}{\ell^2}\|(\Delta G)_\ell\|_2 \|u\|_{L^\infty([0,T]; L^2(\mathbb{T}^d))},
\label{point-est1} \end{eqnarray}
by Young's convolution inequality. The pressure-gradient term $\nabla \ol{(p^\nu)}_{\ell}(x,t)$ in (\ref{weak}) is determined using $\nabla\cdot f^\nu=0$ from the Poisson 
equation 
\be -\Delta \nabla \ol{(p^\nu)}_{\ell}(\cdot,t)= (\nabla\otimes\nabla\otimes\nabla):\ol{(\bu^\nu \otimes \bu^\nu)}_{\ell}(\cdot,t) \ee
and the righthand-side belongs to $C^\infty(\mathbb{T}^d)$ for a.e. time $t$ {and is} bounded above by a constant of the form 
$(1/\ell^3)\|((\nabla \otimes\nabla\otimes \nabla)G)_\ell\|_\infty \|u(\cdot,t)\|^2_{L^2(\mathbb{T}^d)}.$ 
The solution of the Poisson problem thus satisfies a similar estimate as (\ref{point-est1}), i.e. for some constant $C$
and every $x\in \mathbb{T}^d$: 
\be 
\|\nabla \ol{(p^\nu)}_{\ell} (x,\cdot)\|_{L^\infty([0,T])}
\leq \frac{C}{\ell^3}
\|((\nabla \otimes\nabla\otimes \nabla)G)_\ell\|_\infty \|u\|^2_{L^\infty([0,T]; L^2(\mathbb{T}^d))}.
\ee 
We thus see that, except for $\ol{(f^\nu)}_{\ell}(x,\cdot)$, every term in (\ref{weak}) for the distributional 
derivative $\partial_t \ol{(\bu^\nu)}_{\ell}(x,\cdot)$ belongs to $L^\infty([0,T]).$ Since we assume that 
$f^\nu\in L^2([0,T];L^2(\mathbb{T}^d))$, we have for every  $x\in \mathbb{T}^d$ at least:
\be \| \ol{(f^\nu)}_{\ell} (x,\cdot)\|_{L^2([0,T])} \leq \|G_\ell\|_2 \|f^\nu\|_{L^2([0,T];L^2(\mathbb{T}^d))}.\ee 
It follows from Eq. \eqref{weak} that $\partial_t \ol{(\bu^\nu)}_{\ell}(x,\cdot) \in L^2([0,T]),$  so that 
$\ol{(\bu^\nu)}_{\ell}(x,\cdot)$ for every $x\in\mathbb{T}^d$ is absolutely continuous in time 
and the classical time-derivative exists and is given by Eqn. (\ref{weak}) for a.e. $t\in [0,T].$

Taking the Euclidean inner product of (\ref{weak})  with $\overline{(\bu^\nu)}_{\ell}(x,\cdot)$ for each $x\in\mathbb{T}^d$
and writing $\ol{(\bu^\nu \otimes \bu^\nu)}_{\ell}$ $=\ol{(\bu^\nu)}_\ell \otimes \ol{(\bu^\nu)}_{\ell}+{\tau}_{\ell} (\bu^\nu,\bu^\nu)$
yields by the Leibniz product rule the ``resolved energy" balance:
\be
\partial_t\left( \frac{1}{2} |\overline{(\bu^\nu)}_{\ell}|^2 \right) +\nabla \cdot J_{\ell}^\nu =  -\Pi_{\ell}[\bu^\nu] -\nu |\nabla \ol{(\bu^\nu)}_{\ell}|^2 + \ol{(\bu^\nu)}_{\ell}\cdot \ol{(f^\nu)}_{\ell},\ee
with 
\be
\quad \quad  J_{\ell}^\nu:= \left(\frac{1}{2} |\overline{(\bu^\nu)}_{\ell}|^2+ \overline{(p^\nu)}_{\ell} \right)\overline{(\bu^\nu)}_{\ell}  + \overline{(\bu^\nu)}_{\ell}\cdot  {\tau}_{\ell} (\bu^\nu,\bu^\nu) -\nu\nabla \left(\frac{1}{2} |\overline{(\bu^\nu)}_{\ell}|^2\right) ,
\ee
which, again, holds for every $x\in {\mathbb T}^d$ and a.e. $t\in [0,T]$ (and thus distributionally in space-time as well). 
Since  $|\ol{(\bu^\nu)}_{\ell}|^2(x,\cdot)/2$ is absolutely continuous in time, upon integrating we have:
\bea\label{resolvedEnergyBalint}
&&\frac{1}{2} |\overline{(\bu^\nu)}_{\ell}(x,T)|^2 -\frac{1}{2} |\overline{(\bu_0^\nu)}_{\ell}(x)|^2   
=\int_0^T\Big[ -\nabla \cdot J_{\ell}^\nu   -\Pi_{\ell}[\bu^\nu] -\nu |\nabla \ol{(\bu^\nu)}_{\ell}|^2 +\ol{(\bu^\nu)}_{\ell}\cdot \ol{(f^\nu)}_{\ell}\Big](x,t)\ dt \cr
&& 
\eea
for every $T\geq 0$ and $x\in \mathbb{T}^d$.  Since Leray solutions satisfy $u^\nu\in  L^3([0,T]; L^3(\mathbb{T}^d))$ and, consequently, 
$p^\nu\in  L^{3/2}([0,T]; L^{3/2}(\mathbb{T}^d))$ (see e.g. Proposition 1 of \cite{DR00}), each term of the 
integrand inside the square brackets in \eqref{resolvedEnergyBalint} is easily checked 
by the definitions (\ref{J-def}),(\ref{Piell}) to belong to $L^1([0,T];L^1(\mathbb{T}^d))$. The Fubini theorem then gives 
that $\int_{\mathbb{T}^d}\int_0^T \nabla \cdot J_{\ell}^\nu \,\rmd t\,\rmd x=\int_0^T\int_{\mathbb{T}^d} \nabla \cdot J_{\ell}^\nu \,\rmd x\,\rmd t=0$
by space-periodicity, so that integrating (\ref{resolvedEnergyBalint}) over $\mathbb{T}^d,$ we obtain the global balance of 
resolved energy:
\bea\nonumber
&& \frac{1}{2} \int_{\mathbb{T}^d} \! | \ol{(\bu^\nu)}_\ell(x,T)|^2\rmd \bx - \frac{1}{2} \int_{\mathbb{T}^d} \!  
|\overline{(\bu_0)}_{\ell}(x)|^2\rmd \bx + \int_0^T\! \int_{\mathbb{T}^d} \! \Pi_{\ell}[\bu^\nu]\ \rmd \bx  \rmd t \\
&& \hspace{50pt} +\ \int_0^T\! \int_{\mathbb{T}^d} \! \nu |\nabla \ol{(\bu^\nu)}_{\ell}|^2 \ \rmd \bx  \rmd t
- \int_0^T\! \int_{\mathbb{T}^d} \! \ol{(\bu^\nu)}_{\ell} \cdot \ol{(f)}_{\ell}\ \rmd \bx  \rmd t =0.\label{coarseGained}
\eea

We now show that any Leray solution satisfies the global energy balance (\ref{viscousDiss}) for almost every $T\geq 0$.  
Duchon \& Robert \cite{DR00} prove a local version of \eqref{viscousDiss}, i.e. they show that Leray solutions satisfy 
\be\label{DRloc}
\partial_t \left(\frac{1}{2}|u^\nu|^2 \right)+ \nabla \cdot \left[\left(\frac{1}{2}|u^\nu|^2 +p^\nu\right)u^\nu 
- \nu \nabla \left(\frac{1}{2}|u^\nu|^2 \right)\right] = -\varepsilon[u^\nu] + \bu^\nu \cdot f
\ee
in the sense of distributions on space-time.  We smear \eqref{DRloc} with a test function of the form $\varphi^\epsilon(x,t)=
\psi^\epsilon(t)\chi_{\mathbb{T}^d}(x)$, where $\psi^\epsilon(t)$ approximates the characteristic function of the time-interval $[0,T]$ 
and $\chi_{\mathbb{T}^d}(x)$ is the characteristic function of the whole torus (the constant function 1). This yields: 
\be\label{rightCont}
-\int_0^\infty \psi^{\epsilon\prime}\left(\int_{\mathbb{T}^d} \frac{1}{2}|u^\nu|^2 dx\right)dt 
= -\int_0^\infty\psi^\epsilon\int_{\mathbb{T}^d}   \ve[u^\nu]dxdt  + \int_0^\infty \psi^\epsilon \int_{\mathbb{T}^d}  \bu^\nu \cdot f\  dxdt.
\ee
%This follows since $u^\nu \in L^\infty([0,T], L^2)$ implies that $u^\nu$ can be altered on a zero measure 
%set of times so that $u^\nu \in C_w([0,T], L^2)$ (see e.g. Lemma 8 of \cite{LS10}).     
Recall that Leray solutions $u^\nu$ are right-continuous in time, strongly in $L^2(\mathbb{T}^d)$, for a.e. $t\geq 0$ and, in particular, 
at $t = 0,$ as a consequence of the energy inequality (see Remark 2 of \cite{H88}). To make use of this one-sided continuity, let 
$0\leq \psi^\epsilon(t)\leq 1$ be supported on the interval $[0,T+\epsilon]$ and equal to 1 on $[\epsilon,T]$.  The derivative 
${\psi^{\epsilon\prime}}(t)$ gives the difference of two bump functions, one supported on $[T,T+\epsilon]$ and the other 
supported on $[0,\epsilon]$.   Taking $\epsilon\to 0$ we obtain by the right-continuity that:
\be \label{timeconv}
-\int_0^\infty \psi^{\epsilon\prime}\left(\int_{\mathbb{T}^d} \frac{1}{2}|u^\nu|^2 dx\right)dt \to 
\int_{\mathbb{T}^d} \frac{1}{2}|u^\nu(x,T)|^2 dx - \int_{\mathbb{T}^d} \frac{1}{2}|u^\nu_0(x)|^2 dx, \qquad {\rm a.e.} \ \ \  T\geq 0. 
\ee
The assumption $f^\nu \in  L^2([0,T];L^2(\mathbb{T}^d))$, {\it a-priori} estimate 
$\bu^\nu\in L^\infty([0,T];L^2(\mathbb{T}^d)) \cap L^2([0,T]; H^1(\mathbb{T}^d))$ and the fact that $D[\bu^\nu]$ is a Radon measure
permit the dominated convergence theorem to be applied to guarantee that as $\epsilon\to 0$
\be
%&& 
-\int_0^\infty\psi^\epsilon\int_{\mathbb{T}^d}   \ve[u^\nu]dxdt  + \int_0^\infty \psi^\epsilon \int_{\mathbb{T}^d}  \bu^\nu \cdot f\  dxdt 
%\cr
%&& 
\to -\int_0^T \int_{\mathbb{T}^d}   \ve[u^\nu]dxdt  + \int_0^T \int_{\mathbb{T}^d}  \bu^\nu \cdot f\  dxdt. \ee 
%\eea
Thus, the global energy balance \eqref{viscousDiss} is proved.
 
Adding to \eqref{viscousDiss} the resolved energy balance (\ref{coarseGained}) gives, for almost every $T\geq 0$,
\bea \nonumber
\int_0^T\int_{\mathbb{T}^d} \! \ve[\bu^\nu] \  \rmd \bx \rmd t
&=&  \int_0^{T}\!\int_{\mathbb{T}^d} \!  \Pi_{\ell}[\bu^\nu]\  \rmd \bx \rmd t+ \int_{\mathbb{T}^d} \! \nu |\nabla \ol{(\bu^\nu)}_{\ell}|^2 \ \rmd \bx  \rmd t\\ \nonumber
&&-\   \frac{1}{2} \int_{\mathbb{T}^d} \! \left( | \bu^\nu(\cdot,T)|^2-|\overline{(\bu^\nu(\cdot,T))}_{\ell}|^2 \right) \rmd \bx+   \frac{1}{2} \int_{\mathbb{T}^d} \! \left( | \bu_0|^2-|\overline{(\bu_0)}_{\ell}|^2 \right) \rmd \bx\\ \nonumber
&& + \int_0^{T}\!\int_{\mathbb{T}^d} \!  (\bu^\nu\cdot f- \ol{(\bu^\nu)}_{\ell}\cdot \ol{(f)}_{\ell})\  \rmd \bx \rmd t.
\eea
Since, for integrable $g\in L^1(\mathbb{T}^d)$ one has $\int_{\mathbb{T}^d}  \ol{g}_\ell(x) dx 
=\int_{\mathbb{T}^d} g(x) dx$, we arrive at identity \eqref{fourFifthsLaw}.  
\end{proof}

\begin{proof}[Proof of Lemma \ref{lemma1}]

We first prove the upper bound on the total dissipation of Leray solutions.  By Lemma \ref{lemma}, the global energy dissipation is given by the formula 
\eqref{fourFifthsLaw}. Note that $|\ol{(u^\nu)}_\ell|^2\leq \ol{(|u^\nu|^2)}_\ell$ by convexity and thus the 
contribution from $\tau_{\ell}(u^\nu(\cdot,T); u^\nu(\cdot,T))\geq 0$  in \eqref{fourFifthsLaw} is non-positive and we may drop it at the expense of an inequality:
\bea \nonumber
\int_0^T\int_{\mathbb{T}^d} \! \varepsilon[\bu^\nu] \  \rmd \bx \rmd t
&\leq&  \int_0^{T}\!\int_{\mathbb{T}^d} \!  \Pi_{\ell}[\bu^\nu]\  \rmd \bx \rmd t+  \int_0^{T}\!\int_{\mathbb{T}^d} \! \nu  |\nabla \ol{(\bu^\nu)}_\ell|^2\  \rmd \bx \rmd t\\ 
&&+\    \frac{1}{2} \int_{\mathbb{T}^d} \! \tau_{\ell}(u_0^\nu; u_0^\nu)\ \rmd \bx +\ \int_0^{T}\!\int_{\mathbb{T}^d} \!  \tau_{\ell}(\bu^\nu; f^\nu)\  \rmd \bx \rmd t. \label{fourFifthsLaw1}
\eea
The inequality \eqref{fourFifthsLaw1} then implies:
\bea \label{fourFifthsLawbnd}
\int_0^T\int_{\mathbb{T}^d} \! \varepsilon[\bu^\nu] \  \rmd \bx \rmd t \ \leq 
 \int_0^{T}\!\! \| \Pi_{\ell}[\bu^\nu] \|_{1} \rmd t+  
 \int_0^{T}\!\! \! \nu\|\nabla \ol{(\bu^\nu)}_\ell \|_{2}^2 \ \rmd t + \frac{1}{2}\| \tau_{\ell}(u_0^\nu; u_0^\nu)\|_1 +  \int_0^{T}\!\! \|\tau_{\ell}(\bu^\nu; f^\nu)\|_1dt.\,\,\label{mainIneqFterm}
\eea
The energy flux-through-scale is bounded using the  Constantin--E--Titi commutator estimate \cite{CET94}:
\bea
 \int_0^T\!  \| \Pi_{\ell}[\bu^\nu(t)] \|_{1}   \rmd t  \leq C_G \ell^{3\sigma-1} \int_0^T \|\bu^\nu(t)\|_{B_3^{\sigma,\infty}(\mathbb{T}^d)}^3dt= O(\ell^{3\sigma-1}).\label{CETflux2}
\eea
{Above, $C_G$ is a constant depending on $G$ but not on $\ell$, $\nu$ and 
the ``big-$O$'' notation denotes an upper bound with a constant prefactor depending only upon $G$ and $u$.}  
Next, using the nesting property $L^p(\mathbb{T}^d)\subseteq L^q(\mathbb{T}^d)$, $p\geq q,$
we bound the resolved energy dissipation term
\bea
  \int_0^{T} \nu\|\nabla \ol{(\bu^\nu)}_\ell \|_{2}^2 \ \rmd t  \leq  \int_0^{T} \nu\|\nabla \ol{(\bu^\nu)}_\ell \|_{3}^2 \ \rmd t
  \le {C'_G} \nu \ell^{2(\sigma-1)} \int_0^T \|\bu^\nu(t)\|_{B_3^{\sigma,\infty}}^2dt = O(\nu \ell^{2(\sigma-1)}).\label{viscEst}
  \eea
The remaining terms in (\ref{fourFifthsLawbnd}) are bounded using estimates for coarse-graining cumulants (see, e.g. \cite{CET94,DE17}):
 \bea
\label{Otherest1}
\|  \tau_{\ell}(u_0^\nu; u_0^\nu)\|_1& \le & {C''_G} 
\ell^{2\sigma}  \sup_{\nu>0} \|\bu_0^\nu\|_{B_2^{\sigma,\infty}(\mathbb{T}^d)}^2 = O(\ell^{2\sigma}  ),\\  
   \int_0^{T}\|\tau_{\ell}(\bu^\nu;  f^\nu)\|_1dt %&\leq& \sup_{|r|<\ell} \| \delta f^\nu(r)\|_{L^{2}([0,T];L^2)} \  \sup_{|r|<\ell}  \|\delta \bu^\nu(r) \|_{L^{3}([0,T];L^3)} \\
    & \le &  {C''_G}\ell^{2\sigma} \sup_{\nu>0} \| f^\nu \|_{L^{2}([0,T];B_2^{\sigma,\infty}(\mathbb{T}^d))}   
    \sup_{\nu>0} \|\bu^\nu\|_{L^{3}([0,T];B_3^{\sigma,\infty}(\mathbb{T}^d))}  = O(\ell^{2\sigma}).\label{Otherest2}
 \eea  
Thus, combining the estimates (\ref{CETflux2}), (\ref{viscEst}), (\ref{Otherest1}) and (\ref{Otherest2}) in the inequality (\ref{mainIneqFterm}), we find that:
\bea 
\int_0^T\int_{\mathbb{T}^d} \! \varepsilon[\bu^\nu] \  \rmd \bx \rmd t=O(\ell^{3\sigma-1})+ O(\nu \ell^{2(\sigma-1)}).
\label{mainIneqFtermBND}
\eea
Here a term $O(\ell^{2\sigma})$ has been absorbed into $O(\ell^{3\sigma-1}),$ since for $\sigma\leq 1$ it is always smaller as $\ell\to 0.$
Because $\ell>0$ in (\ref{mainIneqFtermBND}) is arbitrary, we specify a relation between $\ell$ and $\nu$ which optimizes the upper bound 
by balancing the contribution of the non-linear flux with the resolved dissipation.  This fixes a relationship $\ell\sim\nu^{1/(\sigma+1)}$ 
and yields the final upper bound:
\bea \nonumber
\int_0^T\int_{\mathbb{T}^d} \! \varepsilon[\bu^\nu] \  \rmd \bx \rmd t  = O(\nu^{\frac{3\sigma-1}{\sigma+1}})
\label{mainIneqFtermBND2}
\eea
as claimed in \eqref{noAnom}. {It is worth remarking that $\ell\sim\nu^{1/(\sigma+1)}$ is the expected scaling in phenomenological 
theory for the ``dissipation length'' where nonlinear energy flux and viscous energy dissipation become comparable, when the velocity increments 
exhibit scaling $\delta u(\ell)\sim \ell^\sigma$. See \cite{PV87,F95}.}  
\end{proof}

\begin{proof}[Proof of Theorem \ref{theorem2}]

We now show under either condition (i) or (ii) that $u$ is a weak solution of the Euler equations which satisfies 
distributionally the local energy balance:
\be \label{limitBalance}
\partial_t\left(\frac{1}{2}|u|^2\right)+\nabla\cdot\left[\left(\frac{1}{2}|u|^2+p\right)u\right] 
= -D[u] + u\cdot f, \qquad D[u] := \stackrel[\ell\to 0]{}{\Dlim}\Pi_{\ell}[\bu]. \ee 
We prove these conclusions separately for condition (i) and for condition (ii):\\

\vspace{-2mm} 
\noindent \emph{Proof of Theorem 2(i)}: 
We apply the Aubin-Lions-Simon Lemma, stated as in Theorem II.5.16 of \cite{BF13}, with $p=3,$ $r=3/2,$
$B_0=B_3^{\sigma,\infty}(\mathbb{T}^d),$ $B_1=L^3(\mathbb{T}^d)$, and $B_2=B^{\sigma-2,\infty}_{3/2}(\mathbb{T}^d).$
The imbedding of $B_3^{\sigma,\infty}(\mathbb{T}^d)$ in $L^3(\mathbb{T}^d)$ is compact by the
Kolmogorov-Riesz theorem and $L^3(\mathbb{T}^d)=F^{0,2}_3(\mathbb{T}^d)${, a Triebel-Lizorkin space 
(see \cite{ST87}, section 3.5),} is continuously embedded in 
$B^{\sigma-2,\infty}_{3/2}(\mathbb{T}^d)$ (e.g. Remark 3.5.1.4, \cite{ST87}).

{
We now show that a distributional Navier-Stokes solution $u\in L^{3}([0,T];B^{\sigma,\infty}_{3}(\mathbb{T}^d))$ has a weak time-derivative 
in the sense of Definition II.5.7 of \cite{BF13}, which is given by 
\be
\label{disttimeder} \frac{du^\nu}{dt}=-{\mathbb P}\nabla\cdot(u^\nu \otimes u^\nu)+\nu\Delta u^\nu + f^\nu 
\in L^{3/2}([0,T];B^{\sigma-2,\infty}_{3/2}(\mathbb{T}^d)), 
\ee
with ${\mathbb P}$ the Leray projector.  To see this, choose smooth test functions of the form $\varphi(t,x)=\psi(t) \phi(x)$ with $\psi\in C_0^\infty((0,T))$ and $\phi\in C^\infty(\mathbb{T}^d,\mathbb{R}^d)$, giving 
\begin{align}
 \left\langle \int_0^T \partial_t \psi(t) u(t) dt, \phi\right\rangle  &=- \left\langle \int_0^T \psi(t) \Big[ -\mathbb{P}\nabla\cdot (u\otimes u)(t) + \nu \Delta u(t)+ f^\nu(t) \Big]dt ,  \phi\right\rangle, \label{weakderiv}
\end{align}
where $\langle \cdot, \cdot \rangle$ denotes the usual pairing between elements of $D'(\mathbb{T}^d)$ and 
$D(\mathbb{T}^d)=C^\infty(\mathbb{T}^d)$.  
We next observe that each term inside the square bracket on the righthand side of the previous equation belongs to 
$L^{3/2}([0,T];B^{\sigma-2,\infty}_{3/2}(\mathbb{T}^d))$ with norms uniformly bounded in $\nu$.} First, by the Calderon-Zygmund inequality we have for some constant $c_0$ {depending only on space dimension $d$} the estimate
\be\|{\mathbb P}\nabla\cdot(u^\nu \otimes u^\nu)\|_{L^{3/2}([0,T];B^{\sigma-2,\infty}_{3/2}(\mathbb{T}^d))}
\leq  c_0\|u^\nu \otimes u^\nu\|_{L^{3/2}([0,T];B^{\sigma-1,\infty}_{3/2}(\mathbb{T}^d))} 
\leq c_0\|u^\nu \|_{L^{3}([0,T];B^{\sigma,\infty}_{3}(\mathbb{T}^d))}^2 . 
\label{ABS1} \ee
On the other hand,
\be\|\Delta u^\nu\|_{L^{3/2}([0,T];B^{\sigma-2,\infty}_{3/2}(\mathbb{T}^d))}
\leq  c_1\|u^\nu\|_{L^{3/2}([0,T];B^{\sigma,\infty}_{3/2}(\mathbb{T}^d))}
\leq  c_1\|u^\nu\|_{L^{3}([0,T];B^{\sigma,\infty}_{3}(\mathbb{T}^d))}.
\label{ABS2} \ee
Finally, because the sequence $f^\nu$ is strongly convergent, it is uniformly bounded in $L^2([0,T];L^2(\mathbb{T}^d))$
and 
\be\| f^\nu\|_{L^{3/2}([0,T];B^{\sigma-2,\infty}_{3/2}(\mathbb{T}^d))}
\leq \|f^\nu\|_{L^{2}([0,T];L^{2}(\mathbb{T}^d))}. \label{ABS3} \ee
{These bounds imply that the element of $D'(\mathbb{T}^d)$ which is paired with $\phi$
on the right side of \eqref{weakderiv} in fact belongs to $B^{\sigma-2,\infty}_{3/2}(\mathbb{T}^d)$. 
Moreover, $ \int_0^T \partial_t \psi(t) u(t) \, dt\in B^{\sigma,\infty}_{3}(\mathbb{T}^d)$ 
on the left side of \eqref{weakderiv}. Since there is the Banach space duality 
$\left(B^{2-\sigma,1}_3(\mathbb{T}^d)\right)'=B^{\sigma-2,\infty}_{3/2}(\mathbb{T}^d)$ 
and $D(\mathbb{T}^d)$ is dense in $B^{2-\sigma,1}_3(\mathbb{T}^d)$
(\cite{ST87}, section 3.5.6), we can extend the relation \eqref{weakderiv} to $\phi\in B^{2-\sigma,1}_3(\mathbb{T}^d)$ 
by continuity and this implies the equality 
\begin{align}
 \int_0^T \partial_t \psi(t) u(t) dt  &=- \int_0^T \psi(t) \Big[ -\mathbb{P}\nabla\cdot (u\otimes u)(t) + \nu \Delta u(t)+ f^\nu(t) \Big]dt, 
\label{weakderiv2}
\end{align}
as elements of $B^{\sigma-2,\infty}_{3/2}(\mathbb{T}^d)$.  It follows 
that \eqref{disttimeder} holds in the sense of Definition II.5.7 of \cite{BF13}.}

By the estimates (\ref{ABS1})-(\ref{ABS3}), one has {furthermore} 
\bea  \left\|\frac{du^\nu}{dt}\right\|_{L^{3/2}([0,T];B^{\sigma-2,\infty}_{3/2}(\mathbb{T}^d))}
\!\!\! \!\!\! \!\!\!  && \leq c_0\|u^\nu \|_{L^{3}([0,T];B^{\sigma,\infty}_{3}(\mathbb{T}^d))}^2
+\nu c_1\|u^\nu\|_{L^{3}([0,T];B^{\sigma,\infty}_{3}(\mathbb{T}^d))} \cr
&& \hspace{80pt} + \|f^\nu\|_{L^{2}([0,T];L^{2}(\mathbb{T}^d))}.
\eea
{In view of our assumptions (i) in Theorem \ref{theorem2},} the family of weak time-derivatives
$\{du^\nu/dt\}_{\nu>0}$ is uniformly bounded in $L^{3/2}([0,T];B^{\sigma-2,\infty}_{3/2}(\mathbb{T}^d))$.
The conditions of the 
Aubin-Lions-Simon Lemma are therefore satisfied, so that $\{u^\nu\}_{\nu>0}$ is relatively compact in 
$L^3([0,T],L^3(\mathbb{T}^d)).$ Subsequences $\nu_k\to 0^+$ thus always exist so that 
$u^{\nu_k}\to u$ strongly in $L^{3}(\mathbb{T}^d\times [0,T])$.  
%A repeat of this argument with $p=2$
%for the family $\{f^\nu\}_{\nu>0}$ implies that the subsequence can be further chosen so that 
%$f^{\nu_k}\to f$ strongly in $L^{2}(\mathbb{T}^d\times [0,T])$. 
For any such subsequence, we can apply 
the arguments of \cite{DR00} to obtain the statements (\ref{Ebal}),(\ref{flux-anom}),(\ref{flux-anom2}).
\\
%\vspace{-2mm} 
\noindent \emph{Proof of Theorem 2(ii)}: 
{
First we show any limit $u$ is a weak Euler solution.  Recall our assumptions \eqref{wkconv}: For $\nu\to 0$
\be u^\nu(\cdot,t)\stackrel[L^3]{}{\rightharpoonup} u(\cdot,t), 
\quad 
(u^\nu \otimes u^\nu)(\cdot,t) \stackrel[L^{3/2}]{}{\rightharpoonup} (u\otimes u)(\cdot,t), \quad f^\nu(\cdot,t)\stackrel[L^2]{}{\rightharpoonup} f(\cdot,t)
\quad   \mbox{ a.e. $t\in [0,T]$}.  \label{Convassumptions}
\ee 
These conditions imply that $\overline{(f^\nu)}_{\ell}\to \overline{(f)}_{\ell},$ $\overline{(\bu^\nu)}_{\ell}\to \overline{(\bu)}_{\ell}$ and $\overline{(\bu^\nu\otimes u^\nu)}_{\ell}\to \overline{(u\otimes u)}_{\ell}$
pointwise in space, a.e. $t$.  Integrating the coarse-grained Navier-Stokes equations \eqref{weak} against an arbitrary solenoidal test function $\varphi\in D([0,T]\times \mathbb{T}^d)$ yields:
\bea
-\langle\partial_t\varphi, \overline{(\bu^\nu)}_{\ell} \rangle  &=& \langle\nabla\varphi,  \ol{(\bu^\nu \otimes \bu^\nu)}_{\ell}\rangle+   \nu\langle\Delta\varphi,  \overline{(\bu^\nu)}_{\ell}\rangle+ \langle\varphi,  \ol{(f^\nu)}_{\ell}\rangle.
\label{resolvedweak}
\eea
 To show convergence {as $\nu\to 0$},  we obtain uniform bounds for all the 
 {integrands} in \eqref{resolvedweak} 
 and apply Lebesgue dominated convergence.
{Such bounds} are easily obtained by applying Young's inequality for convolutions:
\be\label{b1}
| \ol{(u^\nu)}_\ell(x,t)| \leq  \|G_\ell\|_{3/2} \|u^\nu(\cdot,t)\|_{3} \lesssim \|u^\nu(\cdot,t)\|_{3},
\ee
\be \label{b2}
|\ol{(\bu^\nu \otimes \bu^\nu)}_{\ell}(x,t)|\leq \|G_\ell\|_3 \|u^\nu\otimes u^\nu(\cdot,t)\|_{3/2}  \lesssim  \|u^\nu(\cdot,t)\|_{3}^2,
\ee
\be\label{b3}
| \ol{(f^\nu)}_\ell(x,t)| \leq  \| G_\ell\|_{2} \|f^\nu(\cdot,t)\|_{2} \lesssim \|f^\nu(\cdot,t)\|_{2},
\ee
{where the notation $\lesssim$ indicates an upper bound with constant prefactor depending 
on $G$ and $\ell$, but not on $\nu$.}
By our assumption $u^\nu\in  L^3([0,T]; L^3(\mathbb{T}^d))$ and $f^\nu\in  L^2([0,T]; L^2(\mathbb{T}^d))$ with norms uniformly bounded,  all of the upper bounds \eqref{b1}--\eqref{b3} are in $L^1(\mathbb{T}^d\times [0,T])$ uniformly in $\nu>0$. 
Note that the term in \eqref{resolvedweak} with viscosity as a pre-factor vanishes as $\nu\to 0$
\bea
\nu\langle\Delta\varphi,  \overline{(\bu^\nu)}_{\ell}\rangle \leq \nu  \| {\Delta} \varphi \|_2  \|u^\nu\|_{L^\infty([0,T];L^2(\mathbb{T}^d))}&\stackrel[\nu\to 0]{}{\longrightarrow} & 0.
\eea
We may therefore apply dominated convergence to obtain
from (\ref{resolvedweak})  for fixed $\ell> 0$ that in the limit $\nu\to 0$
\bea\nonumber
-\langle\partial_t\varphi, \overline{\bu}_{\ell} \rangle  &=& \langle\nabla\varphi,  \ol{(\bu \otimes \bu)}_{\ell}\rangle+ \langle\varphi,  \ol{f}_{\ell}\rangle 
\label{resolvedweak2}
\eea
The argument is completed by taking the limit $\ell\to 0$, using the fact that mollification can be removed strongly in $L^p$. Taking the limit of equation \eqref{resolvedweak2} thus shows that $u$ is a weak Euler solution.
}

The energy balance \eqref{limitBalance} {is proved by a very similar argument.}
Smearing the resolved energy balance \eqref{resolvedEnergyBal} {established in Lemma \ref{lemma}} with an arbitrary test function $\varphi\in D([0,T]\times \mathbb{T}^d)$ yields:
\bea\nonumber
-\langle\partial_t\varphi, \frac{1}{2} |\overline{(\bu^\nu)}_{\ell}|^2 \rangle  &=& \langle\nabla\varphi,   J_{\ell}^0[u^\nu]\rangle- \langle\Delta\varphi,   \frac{\nu}{2} |\overline{(\bu^\nu)}_{\ell}|^2\rangle \\
&& +\langle\varphi,-\Pi_{\ell}[\bu^\nu]-\nu |\nabla \ol{(\bu^\nu)}_{\ell}|^2+\ol{(\bu^\nu)}_{\ell}\cdot \ol{(f^\nu)}_{\ell}\rangle
\label{resolvedEnergyBal2}
\eea
where $J_{\ell}^0[u^\nu]$ is the inviscid part of the energy current $J_{\ell}[u^\nu]$ defined in (\ref{J-def}), or 
$$
J_{\ell}^0[u^\nu]:= \left(\frac{1}{2} |\overline{(\bu^\nu)}_{\ell}|^2+ \overline{(p^\nu)}_{\ell} \right)\overline{(\bu^\nu)}_{\ell}  + \overline{(\bu^\nu)}_{\ell}\cdot  
{\tau}_{\ell} (\bu^\nu,\bu^\nu). 
$$
First note that the terms involving viscosity as a pre-factor vanish pointwise in space-time:
\bea
\nu |\nabla \ol{(\bu^\nu)}_{\ell}(x,t)|^2\leq \frac{\nu}{\ell^2} \|(\nabla G)_\ell\|_2^2  \|u^\nu\|_{L^\infty([0,T];L^2(\mathbb{T}^d))}^2&\stackrel[\nu\to 0]{}{\longrightarrow} & 0, \\
\quad  \quad \frac{\nu}{2} |\overline{(\bu^\nu)}_{\ell}(x,t)|^2 \leq \frac{\nu}{2} \|G_\ell\|_2^2 \|u^\nu\|_{L^\infty([0,T];L^2(\mathbb{T}^d))}^2  &\stackrel[\nu\to 0]{}{\longrightarrow} & 0.
\eea
The above bounds follow from Young's inequality for convolutions.  Thus, the contribution from these terms will vanish 
in (\ref{resolvedEnergyBal2}) for $\nu\to 0$ and we must now argue that the remaining terms converge.  

In addition to the pointwise-in-$x$ convergence of the mollified quantities discussed above, we have similarly that $\tau_\ell(u^\nu,u^\nu)\to \tau_\ell(u,u)$ pointwise in space for a.e. $t$.  Moreover, by general theory of Calder\'on-Zygmund operators, the map $u^\nu \otimes u^\nu\to p^\nu$ is strongly continuous in 
$L^p(\mathbb{T}^d)$ for $p\in (1,\infty)$ (see e.g. \cite{DR00}).  
In particular, for $p=3/2,$ the assumption on weak convergence of $u^\nu \otimes u^\nu$ in \eqref{Convassumptions} implies that $p^\nu\rightharpoonup p$ 
weakly in $L^{3/2}(\mathbb{T}^d)$ a.e. $t$. Thus, all of the following terms converge pointwise in space, for a.e. $t$:
\be
\frac{1}{2} |\overline{(\bu^\nu)}_{\ell}|^2\to \frac{1}{2} |\overline{\bu}_{\ell}|^2, \quad J_{\ell}^0[u^\nu]\to  J_{\ell}^0[u], \quad \Pi_{\ell}[\bu^\nu]\to \Pi_{\ell}[\bu] ,\quad  \ol{(\bu^\nu)}_{\ell}\cdot \ol{(f^\nu)}_{\ell}\to \ol{\bu}_{\ell}\cdot \ol{f}_{\ell}
\ee
since they are made up of products of objects which converge pointwise.  

Once again, convergence in the sense of distributions follows if integrable bounds can be obtained that allow us to infer 
limits of the smeared terms in (\ref{resolvedEnergyBal2}) by dominated convergence. 
Recall by our assumptions that $u^\nu\in  L^3([0,T]; L^3(\mathbb{T}^d))$ and $p^\nu\in  L^{3/2}([0,T]; L^{3/2}(\mathbb{T}^d))$ 
not only for each $\nu>0$ (as holds for every Leray solution) but also with norms bounded uniformly in $\nu>0$.   
Using Young's inequality for convolutions and H\"{o}lder's inequality, we have pointwise in space-time:
\be
|\nabla \ol{(u^\nu)}_\ell(x,t)| \leq \frac{1}{\ell} \|(\nabla G)_\ell\|_{3/2} \|u^\nu(\cdot,t)\|_{3} \lesssim \|u^\nu(\cdot,t)\|_{3}\ee
\be 
|\tau_{\ell}(u^\nu, u^\nu)(x,t)|\leq \|G_\ell\|_3 \|(u^\nu\otimes u^\nu)(\cdot,t)\|_{3/2}  
+ \|G_\ell\|_{3/2}^2 \|u^\nu(\cdot,t)\|_{3}^2 \lesssim  \|u^\nu(\cdot,t)\|_{3}^2. 
\ee
Likewise we have for the terms appearing in (\ref{resolvedEnergyBal2}) that 
\bea\nonumber
\frac{1}{2} |\overline{\bu^\nu(x,t)}_{\ell}|^2 \!\! \! &\lesssim &\!\! \!  \|u^\nu(\cdot,t)\|_{2}^2, 
\qquad |J_{\ell}^0[u^\nu](x,t)]| \,\lesssim \,  \|u^\nu(\cdot,t)\|_{3}^3 + \|p^\nu(\cdot,t)\|_{3/2}\|u^\nu(\cdot,t)\|_{3},\\
  |\Pi_{\ell}[\bu^\nu](x,t)]| \!\! \!  &\lesssim & \!\! \! \|u^\nu(\cdot,t)\|_{3}^3,
  \quad\quad\  |\ol{(\bu^\nu)}_\ell(x,t)\cdot \ol{(f^\nu)}_\ell(x,t)|\lesssim \|u^\nu(\cdot,t)\|_{2}\|f^\nu(\cdot,t)\|_{2},
\eea
Since all of the latter upper bounds are in $L^1(\mathbb{T}^d\times [0,T])$ uniformly in $\nu>0$ under 
our assumptions, we can apply dominated convergence theorem to obtain
from (\ref{resolvedEnergyBal2})  for fixed $\ell> 0$ that in the limit $\nu\to 0$
\bea\label{resolvedEnergyBalLim}
&&\partial_t\left( \frac{1}{2} |\overline{\bu}_{\ell}|^2 \right) +\nabla \cdot J_{\ell}^0[u] =  -\Pi_{\ell}[\bu] + \ol{\bu}_{\ell}\cdot \ol{f}_{\ell}, 
\eea
in the sense of space-time distributions. We note in particular that 
\be\label{efluxlin} \stackrel[\nu\to 0]{}{\Dlim}\Pi_{\ell}[\bu^\nu]= \Pi_{\ell}[\bu]:= -\nabla\ol{(u)}_\ell:\tau_\ell(u,u). \ee 
%{We remark that the formula \eqref{efluxlin} for the the energy flux $\Pi_{\ell}[\bu]$ can,  using the symmetry of the turbulent or sub-grid stress $\tau_\ell$,  equivalently be written as $-S[\ol{\bu}_\ell]:\tau_\ell(u,u)$ where $S[\bu] = \frac{1}{2}((\nabla u) +  (\nabla u)^T)$ is the strain matrix. This latter form shows that $\Pi_{\ell}$ is the deformation work of the large-scale strain against the small-scale stress.}

The argument is completed by taking the limit $\ell\to 0$ of \eqref{resolvedEnergyBalLim} and showing 
that \eqref{limitBalance} holds distributionally. This fact is proved in \cite{DR00} using a somewhat different regularization. 
For all terms except $\Pi_\ell[u]$, distributional convergence follows directly from the strong continuity of shifts in $L^p$ 
since $u\in  L^3([0,T]; L^3(\mathbb{T}^d))$ and $p\in  L^{3/2}([0,T]; L^{3/2}(\mathbb{T}^d))$.  
{In particular, the term $\overline{\bu}_{\ell}\cdot  {\tau}_{\ell} (\bu,\bu)$
in $J^0_\ell[u]$ vanishes by the commutator identity for ${\tau}_{\ell} (\bu,\bu)$ in \cite{CET94}.}
Convergence of the flux $\Pi_\ell[u]$ is then inferred from the distributional equality:
\be -\stackrel[\ell\to 0]{}{\Dlim}\Pi_\ell[u]= \partial_t\left(\frac{1}{2}|u|^2\right)+\nabla\cdot\left[\left(\frac{1}{2}|u|^2+p\right)u\right]- u \cdot f :=D[u]. 
\ee 

{Under condition (i),} the limiting Euler solutions $u\in L^3(\mathbb{T}^d\times [0,T])$ 
have {additional} space-regularity. The uniform boundedness condition 
in (i) of Theorem \ref{theorem2}, $\sup_{\nu>0}\|u^\nu\|_{L^{3}([0,T];B_3^{\sigma,\infty}({\mathbb T}^d))}<\infty,$ implies that 
\be  \|u^{\nu}\|_{L^{3}(\mathbb{T}^d\times [0,T])}< C', \quad 
\|u^{\nu}(\cdot+r,\cdot)-u^{\nu}\|_{L^{3}(\mathbb{T}^d\times [0,T])}< C|r|^\sigma. \label{incBnd}
\ee 
with constants $C,$ $C'$ independent of viscosity.
The inequalities (\ref{incBnd}) are preserved under strong limits  in $L^{3}(\mathbb{T}^d\times [0,T])$ 
and thus the limiting Euler solutions $u$ under condition (i) satisfy them as well. 
{This yields immediately $u\in L^3([0,T],B_3^{\sigma', c_0}(\mathbb{T}^d)$ for any $\sigma'<\sigma,$
with definitions as in Remark \ref{rem:Bc}. Finally, $D[u]=0$ for $\sigma\in (1/3,1]$ follows from the additional 
space-regularity by the results of \cite{CCFS08}.}   
\end{proof} 

\begin{rem}\label{limitBesov} 
{\rm Although not stated in the theorem, the inequalities \eqref{incBnd} are again preserved in the limit if we add to condition (ii)
the assumption that (\ref{incBnd}) holds with constants $C,$ $C'$ independent of viscosity. 
Weak lower-semicontinuity of the $L^3(\mathbb{T}^d)$-norm and of 
\be \|u^\nu(\cdot+r,t)-u^\nu(\cdot,t)\|_3=\sup_{\|w\|_{3/2}=1}|\langle w(\cdot-r)-w,u^\nu(\cdot,t)\rangle|\ee 
and Fatou's lemma in time, together with the assumption (\ref{incBnd}), guarantees that limiting Euler solutions 
$u$ under this strengthened condition (ii) satisfy the same bound. This is analogous to Remark 3.5 in \cite{CV17}.
} \end{rem}

 %%%%%%%%%%% %%%%%%%%%%% %%%%%%%%%%% %%%%%%%%%%% %%%%%%%%%%% %%%%%%%%%%%
\subsection*{Acknowledgments}  We are grateful to Philip Isett and Lazslo Sz\'ekelyhidi Jr. for useful conversations 
during the IPAM ``Turbulent Dissipation, Mixing and Predictability" workshop of January 2017. We also thank
Susan Friedlander and James Glimm for informing us of their earlier related results.  
%{and Stavros Tavoularis for making us aware of his experimental work measuring weakly vanishing dissipation rates}.  
Research of TD is supported by NSF-DMS grant 1703997. The paper was completed during G.E.'s 
participation in the program  ``Geometrical and Statistical Fluid Dynamics,'' October 2017, 
at the Simons Center for Geometry \& Physics, whose funding for his stay is happily acknowledged.
 %%%%%%%%%%% %%%%%%%%%%% %%%%%%%%%%% %%%%%%%%%%% %%%%%%%%%%% %%%%%%%%%%%

\end{document}